\documentclass{amsart}
\usepackage{amsmath,amsthm,amssymb}

\makeatletter
    
    \@addtoreset{equation}{section}
  \makeatother

\newtheorem{definition}{Definition}[section]
\newtheorem{proposition}[definition]{Proposition}
\newtheorem{theorem}[definition]{Theorem}
\newtheorem{lemma}[definition]{Lemma}
\newtheorem{corollary}[definition]{Corollary}
\theoremstyle{definition}
\newtheorem{remark}{Remark}[section]

\newcommand{\R}{\mathbb{R}}

\newcommand{\N}{\mathbb{N}}
\newcommand{\ep}{\varepsilon}
\newcommand{\pa}{\partial}

\pagebreak

\title[Damped wave equation with space-dependent damping]{
Diffusion phenomena for the wave equation with space-dependent damping
in an exterior domain
}

\author[M. Sobajima]{Motohiro Sobajima}
\email[M. Sobajima]{msobajima1984@gmail.com}
\address[M. Sobajima]{Department of Mathematics, Tokyo University of Science, 
1-3 Kagurazaka, Shinjuku-ku, 162-8601, Tokyo, Japan}

\author[Y. Wakasugi]{Yuta WAKASUGI}
\email[Y. Wakasugi]{yuta.wakasugi@math.nagoya-u.ac.jp}
\address[Y. Wakasugi]{Graduate School of Mathematics, Nagoya University,
Furocho, Chikusaku, Nagoya 464-8602 Japan}

\begin{document}
\begin{abstract}
In this paper, we consider the asymptotic behavior of solutions to the
wave equation with space-dependent damping in an exterior domain.
We prove that when the damping is effective,
the solution is approximated by that of the corresponding heat equation
as time tends to infinity.
Our proof is based on semigroup estimates for the corresponding heat equation
and weighted energy estimates for the damped wave equation. 
The optimality of the decay late for solutions is also established. 
\end{abstract}

\maketitle
\section{Introduction}
\footnote[0]{2010 Mathematics Subject Classification. 35L20;35C06;47B25;35B40}
Let
$\Omega \subset \R^N$ $(N\geq 2)$
be an exterior domain with smooth boundary.
We consider the initial-boundary value problem to the wave equation with
space-dependent damping
\begin{align}
\label{dw}
	\left\{\begin{array}{ll}
	u_{tt}-\Delta u + a(x)u_t = 0,& x\in \Omega,\ t>0,
\\
	u(x,t)=0,&x\in \partial \Omega, \ t>0,
\\
	(u,u_t)(x,0) = (u_0, u_1)(x),&x\in \Omega.
	\end{array}\right.
\end{align}
Here
$u=u(x,t)$
is a real-valued unknown function.
The coefficient of the damping term
$a(x)$
is a radially symmetric function on the whole space
$\mathbb{R}^N$
satisfying
$a\in C^2(\mathbb{R}^N)$
and
\begin{align}
\tag{a0}\label{a0}
	a(x) = a_0 |x|^{-\alpha} + o(|x|^{-\alpha})\quad \text{as }|x|\to \infty
\end{align}
with some
$a_0>0$
and
$\alpha \in [0,1)$,
and then we may assume $0\notin\overline{\Omega}$ without loss of generality.
The initial data
$(u_0, u_1)$
belong to
$[H^3(\Omega)\cap H^1_0(\Omega)]\times [H^2(\Omega)\cap H^1_0(\Omega)]$
with the compatibility condition of second order
and satisfy
${\rm supp\,}(u_0,u_1) \subset \{ x \in \Omega ; |x| < R_0 \}$
with some
$R_0>0$.
Here we recall that
for a nonnegative integer $k$,
with the assumption $a\in C^{\max ( k,2 )}$,
the initial data
$(u_0,u_1) \in H^{k+1} \times H^k$
satisfy the compatibility condition of order $k$ if
$u_p=0$ on $\partial \Omega$ for $p=0,1,\ldots,k$, where
$u_p$ are successively defined by
$u_p=\Delta u_{p-2} - a(x) u_{p-1}\ (p=2,\ldots,k)$.
Then, it is known that \eqref{dw} admits a unique solution
\[
	u \in \bigcap_{i=0}^{k+1}C^{i}([0,\infty);H^{k+1-i}(\Omega))
\]
(see Ikawa \cite[Theorem 2]{Ik68}).

Also, we consider the initial-boundary value problem to the corresponding heat equation
\begin{align}
\label{heat}
	\left\{ \begin{array}{ll}
	v_t - a(x)^{-1} \Delta v = 0,&x\in \Omega, t>0,\\
	v(x,t) = 0,&x\in \partial \Omega, t>0,\\
	v(x,0) = v_0(x),&x\in \Omega.
	\end{array}\right.
\end{align}

Our aim is to prove that the asymptotic profile of the solution to \eqref{dw}
is given by a solution of \eqref{heat} as time tends to infinity.
Namely, the solution of the damped wave equation \eqref{dw}
has the diffusion phenomena.

The diffusive structure of the damped wave equation has been studied
for a long time.
Matsumura \cite{Ma76} proved $L^p$-$L^q$ estimates of solutions
in the case
$\Omega = \R^N$
and
$a(x)\equiv 1$.
On the other hand,
Mochizuki \cite{Mo76}
considered the case where
$\Omega = \R^N$
and the coefficient
$a=a(x,t)$
satisfies
$0\le a(x,t) \le C(1+|x|)^{-\alpha}$
with
$\alpha>1$,
and proved that in general the energy of the solution does not decay to zero.
Moreover, for some initial data, the solution approaches to a solution to
the wave equation without damping in the energy sense.
Mochizuki and Nakazawa \cite{MoNa96} generalized it to
the case of exterior domains with
star-shaped complement. 
Matsuyama \cite{Mat02} further extended it to more general domains
by adding the assumption of the positivity of $a(x,t)$ around $\partial\Omega$.

On the other hand, when the coefficient
$a$
satisfies
$a(x) \ge C (1+|x|)^{-\alpha}$
with some
$\alpha \in [0,1)$,
Matsumura \cite{Ma77} and Uesaka \cite{Ue79}
showed that the energy of the solution decays to zero.
When
$\Omega = \mathbb{R}^N$
and the coefficient
$a(x)$
is radially symmetric and
satisfies \eqref{a0},
Todorova and Yordanov \cite{ToYo09}
introduced a suitable weight function of the form
$t^{-m}e^{\psi}$,
which originates from \cite{ToYo01} and \cite{Ik05IJPAM},
and proved an almost optimal energy estimate
\[
	\int_{\mathbb{R}^N} ( |u_t|^2 + |\nabla u|^2 ) dx
	\le C(1+t)^{-\frac{N-\alpha}{2-\alpha}-1+\ep} \| (u_0,u_1) \|_{H^1\times L^2}^{2}.
\]
After that, Radu, Todorova and Yordanov \cite{RaToYo09}
extended it
to estimates for higher order derivatives.
Nishihara \cite{Ni10} also established a weighted energy method
similar to \cite{ToYo09} and
obtained decay estimates of the solution to the nonlinear problem
\[
	u_{tt}-\Delta u + (1+|x|^2)^{-\alpha/2}u_t + |u|^{p-1}u = 0.
\]
Based on the energy method of \cite{Ni10},
the second author \cite{Wa14} proved the same type estimates
as those obtained in \cite{RaToYo09}
and applied them to the diffusion phenomena for the damped wave equation \eqref{dw}
when
$\Omega = \mathbb{R}^N$
and
$a(x)= (1+|x|^2)^{-\alpha/2}$ with $0\le \alpha <1$.

Ikehata, Todorova and Yordanov \cite{IkToYo13} considered the damping satisfying
$a(x) \ge \mu (1+|x|^2)^{-1/2}$, which corresponds to the case
$\alpha =1$.
They proved that
the energy of the solution decays as $O(t^{-\mu})$ if $1<\mu <N$
and $O(t^{-N+\varepsilon})$ with arbitrary $\varepsilon>0$ if $\mu \ge N$,
respectively.

Recently, Nishiyama \cite{Nis15} studied the abstract damped wave equation
$u_{tt} +Au + Bu_t =0$
and proved the diffusion phenomena
by using the resolvent argument
when the damping term $B$ is strictly positive.
Radu, Todorova and Yordanov \cite{RaToYopre} considered the equation
$Cu_{tt} + Au + u_t = 0$
and obtained a similar result via the method of diffusion approximation.

In this paper, we prove the diffusion phenomena for the damped wave equation \eqref{dw}
in an exterior domain and for more general damping term than \cite{Wa14}.
Moreover, we prove the optimality of the decay rate for the solution of the
corresponding parabolic problem \eqref{heat} in a special case,
which implies that the asymptotic profile of the solution of the damped wave equation \eqref{dw}
is actually given by a solution of the corresponding heat equation \eqref{heat}
(see Proposition \ref{prop_opt}).
Our main result is the following:

\begin{theorem}\label{thm1}
Let
$u$
be a solution of \eqref{dw} with initial data
$(u_0,u_1)$
and let
$v$
be a solution of \eqref{heat} with
$v_0 = u_0 + a(x)^{-1}u_1$.
Then, for any $\ep>0$, there exists $C>0$ such that we have
\begin{align*}
	\| \sqrt{a(\cdot)} ( u(\cdot,t) - v(\cdot,t) ) \|_{L^2(\Omega)}
	\le C (1+t)^{-\frac{N-\alpha}{2(2-\alpha)}-\frac{2-2\alpha}{2-\alpha}+\ep}
		\| (u_0, u_1) \|_{H^2\times H^1(\Omega)}
\end{align*}
for any $t \ge 1$.
\end{theorem}

The proof of the above theorem consists of following three parts.

Firstly, in the next section,
we investigate the heat semigroup
$e^{tL_{\ast}}$
generated by
the Friedrichs extension
$L_{\ast}$
of
$L=a(x)^{-1}\Delta$.
In particular, we prove the semigroup estimate
\[
	\| \sqrt{a(\cdot)} e^{tL_{\ast}} f \|_{L^2}
	\le Ct^{-\frac{N-\alpha}{2(2-\alpha)}} \| a(\cdot) f \|_{L^1}
\]
by using Beurling-Deny criteria (see e.g., Ouhabaz \cite[Section 2]{Ouhabaz}) and 
weighted Gagliardo-Nirenberg inequalities. 
A similar argument can be found in Liskevich and Sobol \cite{LS03}
($L^p$-analysis of $L$ is studied in Metafune and Spina \cite{MetSpi14} 
and Metafune, Okazawa, Spina and the first author \cite{MOSS}.)
Moreover, under some additional assumptions,
we also have the optimality of the above estimate.

Secondly, in Section 3, 
we prove the almost sharp higher-order energy estimate of solutions to
the damped wave equation \eqref{dw}
by using the Todorova-Yordanov-type weight function
\[
	\Phi_{A,\beta}(x,t) = \exp\left( \beta \frac{A(x)}{1+t} \right).
\]
This has also been proved by Radu, Todorova and Yordanov \cite{RaToYo09} when
$\Omega = \mathbb{R}^N$.
In this paper, we give an alternate proof based on the argument of \cite{Ni10}
and a Hardy-type inequality. 

In the final step, in Section 4,
we rewrite the difference of solutions to \eqref{dw} and \eqref{heat} as
\begin{align*}
	u(t) - v(t)&=
		-\int_{t/2}^t e^{(t-s)L_{\ast}} [ a(\cdot)^{-1} u_{ss}(s) ] ds \\
		&\quad - e^{\frac{t}{2}L_{\ast}} [a(\cdot)^{-1} u_t( t/2 )] \\
		&\quad - \int_0^{t/2} \frac{\pa}{\pa s}
			\left( e^{(t-s)L_{\ast}} [ a(\cdot)^{-1}u_s(s) ] \right) ds.
\end{align*}
Applying the heat semigroup estimate for
$e^{tL_{\ast}}$
and the energy estimate for the time-derivatives of $u$,
we prove that the each term of the right-hand side decays faster.

For the end of this section, we introduce the notation used throughout this paper.
The letter $C$ indicates the generic constant, which may change from line to line.
We denote the set of all compactly supported smooth functions in $U$ $(\subset \R^N)$ as $C_c^\infty(U)$ 
and
the $L^p$ norm by $\| \cdot \|_{L^p}$, that is,
\[
	\| f\|_{L^p} =
	\left\{ \begin{array}{ll}
	\displaystyle \left( \int_{\Omega} | f(x) |^p dx \right)^{1/p} &(1\le p <\infty),\\
	\displaystyle {\rm ess\, sup\, } |f(x)| &(p=\infty).
	\end{array} \right.
\]
For a nonnegative integer $k$, $\| \cdot \|_{H^k}$ denotes the Sobolev norm, that is,
\[
	\| f \|_{H^k} = \Bigg(\sum_{|\alpha|\le k} \| \partial_x^k f \|_{L^2}^2\Bigg)^\frac{1}{2}.
\]
For an interval $I$ and a Banach space $X$,
we define
$C^r (I; X)$
as the space of $r$-times continuously differentiable mapping from $I$ to $X$
with respect to the topology in $X$.

\section{The semigroup generated by $a(x)^{-1}\Delta$}
In this section, we study the semigroup generated by the operator
\[
   L
=
   a(x)^{-1}\Delta
\quad 
   \text{in}\ \Omega
\]
endowed with the Dirichlet boundary condition. 
Since the coefficient $a$ is positive and satisfies \eqref{a0}, 
we may assume that there exists $c_0\in (0,1)$ such that
\begin{equation}\label{a0-mod}
   c_0|x|^{-\alpha}
\leq 
   a(x)
\leq 
   c_0^{-1}|x|^{-\alpha},\quad x\in \overline{\Omega}.
\end{equation}
We remark that the results of this section requires only \eqref{a0-mod}
and we do not need that $a$ is radially symmetric and satisfies \eqref{a0}.
We introduce the weighted $L^p$-spaces 
\begin{align*}
   L^p_{d\mu}
:=
   \left\{
      f\in L^p_{\rm loc}(\Omega)
   \;;\;
      \|f\|_{L^p_{d\mu}}:=\int_{\Omega}
         |f(x)|^p a(x)
      \,dx<\infty
   \right\}, 
\quad 
   1\leq p<\infty
\end{align*}
and the bilinear form 
\begin{align*}
\begin{cases}
   \mathfrak{a}(u,v)
:=
   \displaystyle
   \int_{\Omega}
      \nabla u(x)\cdot \nabla v(x)
   \,dx, 
\\[8pt]
   D(\mathfrak{a})
:=
   \Big\{
      u\in C_c^\infty(\overline{\Omega})
   \;;\;
      u(x)=0\quad\forall x\in \pa\Omega
   \Big\}
\end{cases}
\end{align*}
in a Hilbert space $L^2_{d\mu}$. 
Then the form $\mathfrak{a}$ is closable, and therefore, 
we denote $\mathfrak{a}_*$ as a closure of $\mathfrak{a}$. 
Then we can see that
\begin{lemma}
\label{domain.form}
The bilinear form $\mathfrak{a}_*$ can be characterized as follows:
\begin{align}\label{eq:form.dom}
   &D(\mathfrak{a}_*)
=
   \left\{
      u\in L^2_{d\mu}\cap \dot{H}^1(\Omega)
   ;
      \int_{\Omega}
         \frac{\pa u}{\pa x_j}\varphi
      \,dx
   =
      -\int_{\Omega}
         u\frac{\pa \varphi}{\pa x_j}
      \,dx
   \;\;
      \forall \varphi\in C_c^\infty(\R^N)
   \right\},
\\
\label{eq:form}
&\mathfrak{a}_*(u,v)
=
   \int_{\Omega}
      \nabla u(x)\cdot \nabla v(x)
   \,dx.
\end{align}
\end{lemma}
\begin{proof}
We denote $D$ as the function space on the right-hand side of \eqref{eq:form.dom}.
It is clear that $D(\mathfrak{a})\subset D$. 
Since $D$ is closed with respect to the norm 
\[
\|u\|_{D}:=(\|\nabla u\|_{L^2(\Omega)}^2+\|u\|_{L^2_{d\mu}}^2)^\frac{1}{2},
\]
we deduce $D(\mathfrak{a}_*)\subset D$. 

Conversely, let $u\in D$. To prove $u\in D(\mathfrak{a}_*)$,
it suffices to find the sequence $\{u_n\}_n\in D(\mathfrak{a})$ such that
\[
u_n\to u\ \text{in}\ L^2_{d\mu}\quad \text{as}\ n\to \infty, 
\quad
\mathfrak{a}(u_n-u_m,u_n-u_m)\to 0\quad\text{as}\ n,m\to \infty. 
\]
Set a function $\zeta\in C^\infty(\R;[0,1])$ as
\[
   \zeta(s)
=
   \begin{cases}
      1 & {\rm if}\ s<0, 
   \\
      0 & {\rm if}\ s>1 
   \end{cases}
\]
and define a family of cut-off functions 
$\{\zeta_n\}_n\subset C_c^\infty(\R^N;[0,1])$ as 
\[
   \zeta_n(x)
=
   \zeta\big(\log(|x|)-\log R-n\big).
\]
Then we take $u_n(x):=\zeta_n(x)u(x)$ for $n\in\N$.
Observe that 
$u_n\in H^1_0(\Omega)$, $u_n\to u$ in $L^2_{d\mu}$. 
Moreover, for $n,m\in\N$ with $n<m$, 
\begin{align*}
   \int_{\Omega}
      |\nabla (u_n-u_m)|^2
   \,dx
&=
   \int_{\Omega}
      |\nabla ((\zeta_m-\zeta_n)u)|^2
   \,dx
\\
&\leq 
   2
   \int_{\Omega}
      |\zeta_m-\zeta_n|^2\,|\nabla u|^2
   \,dx
   +
   2
   \int_{\Omega}
      |\nabla(\zeta_m-\zeta_n)|^2\,|u|^2
   \,dx.
\end{align*}
Noting that $a$ satisfies \eqref{a0-mod} and then
\[
   |\nabla \zeta_n|
=
   \left|
      \frac{x}{|x|^2}\zeta'\big(\log(|x|)-\log R-n\big)
   \right|
\leq 
   \chi_{\{|x|>Re^{n}\}}|x|^{-1}
\leq 
   C\chi_{\{|x|>Re^{n}\}}a(x)^{\frac{1}{2}}, 
\]
for some constant $C>0$
($\chi_U$ denotes the indicator function of $U$), we obtain 
\[
\int_{\Omega}
|\nabla (u_n-u_m)|^2
\,dx\to 0
\]
as $n,m\to \infty$. Since for every $n\in\N$, 
$u_n$ is compactly supported in $\overline{\Omega}$, 
$u_n$ can be also approximated by functions 
in $C_c^\infty(\overline{\Omega})$ in the sense of $H^1_0(\Omega)$-topology.
This means that $u\in D(\mathfrak{a}_*)$. 
Since \eqref{eq:form} can be verified by the above approximation, 
the proof is completed.
\end{proof}

\begin{lemma}[The Friedrichs extension]\label{L*}
The operator $-L_*$ in $L^2_{d\mu}$ defined by 
\begin{align*}
   D(L_*)
	&:=
   \Big\{
      u\in D(\mathfrak{a}_*)
   \;;\; 
      \exists f\in L^2_{\mu}\text{\ s.t.\ }
            \mathfrak{a}_*(u,v)
         =
            (f,v)_{L^2_{d\mu}}
         \quad
            \forall v\in D(\mathfrak{a}_*)
   \Big\},\\
   -L_*u &:=f
\end{align*}
is nonnegative and selfadjoint in $L^2_{d\mu}$. 
Therefore $L_*$ generates an analytic semigroup $e^{tL_*}$ on $L^2_{d\mu}$ 
and satisfies
\[
   \|e^{tL_*}f\|_{L^2_{d\mu}}
\leq 
   \|f\|_{L^2_{d\mu}}, 
\quad 
   \|L_*e^{tL_*}f\|_{L^2_{d\mu}}
\leq 
   \frac{1}{t}\|f\|_{L^2_{d\mu}}, 
\quad 
   \forall\ f\in L^2_{d\mu}.
\]
Furthermore, $L_{*}$ is an extension of $L$ 
defined on $C_c^\infty(\overline{\Omega})$ 
with Dirichlet boundary condition.
\end{lemma}
\begin{proof}
By \cite[Theorem X.23]{RS2} we see that $-L_*$ 
is nonnegative and selfadjoint in $L^2_{d\mu}$. 
Moreover, for every $u,v\in D(\mathfrak{a})$,
\begin{align*}
   (-Lu,v)_{L^2_{d\mu}}
=
   \int_{\Omega}
      (-Lu)v
   \,d\mu
=
   \int_{\Omega}
      (-\Delta u)v
   \,dx
=
\mathfrak{a}(u,v)
=
\mathfrak{a}_*(u,v).
\end{align*}
Therefore $u\in D(L_*)$ and $L_*u=Lu$.
\end{proof}

\begin{lemma}
\label{domain.op}
We have
\[
\big\{
  u\in H^2(\Omega)\cap H_0^1(\Omega)\;;\;a(x)^{-\frac{1}{2}}\Delta u\in L^2(\Omega)
  \big\}\subset D(L_*)
\] 
and its inclusion is continuous. 
\end{lemma}
\begin{proof} 
Let $u\in H^2(\Omega)\cap H_0^1(\Omega)$ satisfying $a(x)^{-\frac{1}{2}}\Delta u\in L^2(\Omega)$.
By Lemma \ref{domain.form}, we have $u\in D(\mathfrak{a}_*)$. 
Moreover, for every $\varphi\in D(\mathfrak{a})$, we have
\begin{align*}
\left|
   \mathfrak{a}_*(u,\varphi)
\right|
&=
\left|
   \int_{\Omega}\nabla u\cdot\nabla\varphi \,dx
\right|
\\
&=
\left|
\int_{\Omega}(-\Delta u)\varphi\,dx
\right|
\leq 
\|a(\cdot)^{-\frac{1}{2}}\Delta u\|_{L^2}
\|\varphi\|_{L^2_{d\mu}}.
\end{align*}
Therefore $u\in D(L_*)$ and $L_*u=a(x)^{-1}\Delta u$. Moreover, we have
\begin{align*}
\|u\|_{D(L_*)}^2
&=
\|u\|_{L^2_{d\mu}}^2
+\|L_*u\|_{L^2_{d\mu}}^2
\\
&
=\|a(\cdot)^{\frac{1}{2}}u\|_{L^2}^2
+\|a(\cdot)^{-\frac{1}{2}}\Delta u\|_{L^2}^2
\\
&
\leq \|a\|_{L^\infty}\|u\|_{L^2}^2
+\|a(\cdot)^{-\frac{1}{2}}\Delta u\|_{L^2}^2
\\
&
\leq \|a\|_{L^\infty}\|u\|_{H^2}^2
+\|a(\cdot)^{-\frac{1}{2}}\Delta u\|_{L^2}^2. 
\end{align*}
This inequality gives the continuity of the inclusion in the assertion.  
\end{proof}

\begin{lemma}
\label{submarkov}
The semigroup $\{e^{tL_*}\}_{t\geq0}$ given by Lemma \ref{L*}
is sub-markovian, that is, $e^{tL_*}$ is positively preserving:
\begin{align*}
0\leq f\in L^2_{d\mu}\quad \Longrightarrow \quad e^{tL_*}f\geq 0
\end{align*}
and $L^\infty$-contractive:
\begin{align*}
\|e^{tL_*}f\|_{L^\infty}\leq \|f\|_{L^\infty}, 
\quad \forall f\in L^2_{d\mu}\cap L^\infty(\Omega).
\end{align*}
\end{lemma}
\begin{proof}
Let $u\in D(\mathfrak{a}_*)$ be fixed. 
We note that Lemma \ref{domain.form} implies
$|u|, Pu:=(1\wedge |u|){\rm sign}\,u\in D(\mathfrak{a}_*)$. 
Moreover, by direct computation, we have 
\[
\mathfrak{a}_*(|u|,|u|)=
\int_{\Omega}\big|\nabla |u|\big|^2\,dx
=
\int_{\Omega}|\nabla u|^2\,dx
=
\mathfrak{a}_*(u,u)\]
and 
\[
\mathfrak{a}_*(Pu,u-Pu)=
\int_{\Omega}|\nabla u|^2\chi_{\{|u|<1\}}(1-\chi_{\{|u|<1\}})\,dx=0.
\]
Applying \cite[Theorems 2.7 and 2.13]{Ouhabaz}, we respectively obtain 
the positivity and $L^\infty$-contractivity of $e^{tL_*}$.
\end{proof}

\begin{lemma}
[The embedding $D(\mathfrak{a}_*)\hookrightarrow L^q_{d\mu}$]
\label{embedding}
If $N\geq 3$, then there exists $C_{N,\alpha}>0$ such that 
\[
\|u\|_{L^{q_*}_{d\mu}}\leq C_{N,\alpha} \mathfrak{a}_*(u,u)^{\frac{1}{2}}, 
\quad 
\forall u\in D(\mathfrak{a}_*), 
\]
where $q_*:=\frac{2(N-\alpha)}{N-2}>2$. 
If $N=2$, then for every $2<q<\infty$, there exists $C_{2,\alpha,q}>0$ such that 
\begin{align*}
\|u\|_{L^q_{d\mu}}\leq 
C_{2,\alpha,q}\mathfrak{a}_*(u,u)^{\frac{1}{2}-\frac{1}{q}}
\|u\|_{L^2_{d\mu}}^{\frac{2}{q}}, 
\quad 
\forall u\in D(\mathfrak{a}_*). 
\end{align*}
\end{lemma}
\begin{proof}
It suffices to show respective inequalities for $u\in D(\mathfrak{a})$. 
First we prove the case $N\geq 3$. 
Let $u\in D(\mathfrak{a})$. Then noting that 
\[
   q_*=2^*\Big(1-\frac{\alpha}{2}\Big)+\alpha, 
\]
we see from the H\"older inequality that 
\begin{align*}
\|u\|_{L^{q_*}_{d\mu}}^{q_*}
&\leq c_0^{-1}
\int_{\Omega}
|u(x)|^{q_*}|x|^{-\alpha}\,dx
= c_0^{-1}
\int_{\Omega}
\Big(|u(x)|^{2^*}\Big)^{1-\frac{\alpha}{2}}
\Big(|u(x)|^{2}|x|^{-2}\Big)^{\frac{\alpha}{2}}
\,dx
\\
&\leq c_0^{-1}
\Big(\int_{\Omega}
|u(x)|^{2^*}
\,dx\Big)^{1-\frac{\alpha}{2}}
\Big(
\int_{\Omega}
|u(x)|^{2}|x|^{-2}
\,dx
\Big)^{\frac{\alpha}{2}}.
\end{align*}
Using Gagliardo-Nirenberg and Hardy inequalities, we have 
\[
\|u\|_{L^{q_*}_{d\mu}}^{q_*}\leq 
c_0^{-1}C_{GN}^{2^*\theta_1}C_H^{2\theta_2}
\|\nabla u\|_{L^2}^{2^*(1-\frac{\alpha}{2})}
\|\nabla u\|_{L^2}^{\alpha}
=
C'
\|\nabla u\|_{L^2}^{q_*}.
\]
If $N=2$, then set $\Phi(x)=|x|^{-\frac{\alpha}{2}}x$ for $x\in \Omega$. 
Then we see by change of variables that 
\begin{align*}
\|u\|_{L^q_{d\mu}}^q
\leq 
c_0^{-1}
\int_{\Omega}|u(x)|^{q}|x|^{-\alpha}\,dx
=
\frac{2}{(2-\alpha)c_0}\int_{\Phi(\Omega)}\Big|u(\Phi^{-1}(y))\Big|^{q}\,dy.
\end{align*}
Gagliardo-Nirenberg inequality for $v:=u\circ\Phi^{-1}$ implies
\begin{align*}
\|v\|_{L^q(\Phi(\Omega))}\leq C\|\nabla v\|_{L^2(\Phi(\Omega))}^{1-\frac{2}{q}}\|v\|_{L^2(\Phi(\Omega))}^{\frac{2}{q}}.
\end{align*}
Noting that $\Phi^{-1}(y)=|y|^{\frac{\alpha}{2-\alpha}}y$, we have 
\begin{align*}
\int_{\Phi(\Omega)}|\nabla v(y)|^{2}\,dy
&=
\int_{\Phi(\Omega)}|D^*\Phi(y)\nabla u(\Phi^{-1}(y))|^{2}\,dy
\\
&\leq 
\left(\frac{2}{2-\alpha}\right)^2\int_{\Phi(\Omega)}
|\nabla u(\Phi^{-1}(y))|^{2}|y|^{\frac{2\alpha}{2-\alpha}}\,dy
\\
&=
\left(\frac{2}{2-\alpha}\right)^2\int_{\Phi(\Omega)}|\nabla u(\Phi^{-1}(y))|^{2}|\Phi^{-1}(y)|^{\alpha}\,dy. 
\end{align*}
Using change of variables again yields
\begin{align*}
\|\nabla v\|_{L^2(\Phi(\Omega))}^2
&\leq
\frac{2}{2-\alpha}\int_{\Omega}|\nabla u(x)|^{2}\,dx. 
\end{align*}
Combining the inequalities above,  we obtain the desired inequality.
\end{proof}

\begin{proposition}
\label{embedding2}
Let $e^{tL_*}$ be given in Lemma \ref{L*}. 
For every $f\in L^2_{d\mu}$, we have
\begin{equation}
\label{2-infty}
	\|e^{tL_*}f\|_{L^\infty}\leq Ct^{-\frac{N-\alpha}{2(2-\alpha)}}\|f\|_{L^2_{d\mu}}.
\end{equation}
Moreover, for every $f\in L^1_{d\mu}\cap L^2_{d\mu}$, we have
\begin{equation}
\label{1-2}
	\|e^{tL_*}f\|_{L^2_{d\mu}}\leq Ct^{-\frac{N-\alpha}{2(2-\alpha)}}\|f\|_{L^1_{d\mu}}
\end{equation}
and
\begin{align}
\label{1-3}
	\|L_{*} e^{tL_{\ast}}f\|_{L^2_{d\mu}}
		\leq Ct^{-\frac{N-\alpha}{2(2-\alpha)}-1}\|f\|_{L^1_{d\mu}}.
\end{align}
\end{proposition}
\begin{proof}
If $N\geq 3$, then by Lemma \ref{embedding}, we have
for every $t>0$ and $f\in L^2_{d\mu}$, 
\begin{align*}
   \|e^{tL_*}f\|_{L^{q_*}_{d\mu}}
&\leq 
   C_{N,\alpha}\mathfrak{a}_*(e^{tL_*}f,e^{tL_*}f)^{\frac{1}{2}}
\\
&=
   C_{N,\alpha}(-L_*e^{tL_*}f,e^{tL_*}f)_{L^2_{d\mu}}^{\frac{1}{2}}
\\
&\leq 
   C_{N,\alpha}\|L_*e^{tL_*}f\|_{L^2_{d\mu}}^{\frac{1}{2}}\|e^{tL_*}f\|_{L^2_{d\mu}}^{\frac{1}{2}}
\\
&\leq 
   C_{N,\alpha}t^{-\frac{1}{2}}\|f\|_{L^2_{d\mu}}.
\end{align*}
Therefore noting Lemmas \ref{submarkov} and 
applying \cite[Lemma 6.1]{Ouhabaz} with $(\alpha,r)=(\frac{1}{2},q_*)$, 
we obtain \eqref{2-infty}. Moreover, 
for every $f\in L^1_{d\mu}\cap L^2_{d\mu}$ and 
$g\in L^2_{d\mu}$, 
\begin{align*}
\Big|
  (e^{tL_*}f,g)_{L^2_{d\mu}}
\Big|
=
\Big|
  \int_{\Omega}f\big(e^{tL_*}g\big) d\mu
\Big|
\leq 
\|f\|_{L^1_{d\mu}}\|e^{tL_*}g\|_{L^\infty}
\leq 
Ct^{-\frac{N-\alpha}{2(2-\alpha)}}\|f\|_{L^1_{d\mu}}\|g\|_{L^2_{d\mu}}.
\end{align*}
Therefore \eqref{1-2} is proved.
The estimate \eqref{1-3} is easily proved by noting
\[
	\| L_{\ast} e^{tL_{\ast}}f \|_{L^2_{d\mu}}
	= \left\| L_{\ast} e^{\frac{t}{2}L_{\ast}} \left( e^{\frac{t}{2}L_{\ast}} f \right) \right\|_{L^2_{d\mu}}
	\leq \frac{2}{t} \| e^{\frac{t}{2}L_{\ast}} f  \|_{L^2_{d\mu}}
\]
and using \eqref{1-2}.
We can also verify the case of $N=2$ by a similar computation. 
\end{proof}

Finally, we discuss about the optimality of the estimate \eqref{1-2}.
\begin{proposition}
\label{prop_opt}
If $N\geq 2$ and $a(x)=|x|^{-\alpha}$, then there exist 
a function $f\in C_c^\infty(\Omega)$ and constants $\ep>0$ and $t_0>0$ such that 
\[
\|e^{tL_*}f\|_{L^2_{d\mu}}\geq \ep t^{-\frac{N-\alpha}{2(2-\alpha)}}, \quad t\geq t_0.
\]
In other words, the decay rate $t^{-\frac{N-\alpha}{2(2-\alpha)}}$ of $e^{tL_*}$ 
is optimal 
for initial data in $C_c^\infty(\Omega)$.
\end{proposition}
\begin{proof}
First we consider the case $N\geq 3$. 
Let $R$ satisfy $K\subset B(0,R/2)$. Define the function
\[
G(x,t)=
t^{-\frac{N-\alpha}{2-\alpha}}
\left[1-t^{\frac{N-2}{2-\alpha}}|x|^{2-N}\right]
\exp\Big[-\frac{|x|^{2-\alpha}}{(2-\alpha)^2t}\Big], 
\quad x\in \Omega, t> 0
\]
and $G_+(x,t):=\big(G(x,t)\big)_+$.
By \cite[Lemma 5.5]{IMSS} we can see that 
\[
\frac{\pa G}{\pa t}=|x|^{\alpha}\Delta G 
\quad\text{in}\ \Omega=\R^N\setminus \{0\}, 
\]
and 
\[
\begin{cases}
G(x,t)\geq 0&\text{if}\ |x|^{2-\alpha}\geq t,
\\
G(x,t)\leq 0&\text{if}\ |x|^{2-\alpha}\leq t.
\end{cases}
\]
If $g\in L^2_{d\mu}$ satisfies $g\geq G_+(x,t_R)$ with $t_R=R^{2-\alpha}$, 
then since $e^{tL_*}$ is preserves positivity, 
we see that for every $x\in B(0,R)\cap \Omega$, 
\begin{align*}
G(x,t+t_R)-e^{tL_*}g
&\leq G(x,t+t_R)
\leq 0
\end{align*}
and then
\begin{align*}
&
\frac{1}{2}\,\frac{d}{dt}\Big\|\Big(G(x,t+t_R)-e^{tL_*}g\Big)_+\Big\|_{L^2_{d\mu}}^2
\\
&=
\int_{\Omega}
\Big(\frac{\pa G}{\pa t}(x,t+t_R)-\frac{d}{dt}e^{tL_*}g\Big)
\Big(G(x,t+t_R)-e^{tL}g\Big)_+a(x)\,dx
\\
&=
\int_{\R^N\setminus B(0,R)}
\Delta\Big(G(x,t+t_R)-e^{tL_*}g\Big)
\Big(G(x,t+t_R)-e^{tL_*}g\Big)_+\,dx
\\
&=
-\int_{\R^N\setminus B(0,R)}
\Big|\nabla\Big(G(x,t+t_R)-e^{tL}g\Big)_+\Big|^2\,dx
\\
&\leq 0.
\end{align*}
Hence we have 
$e^{tL_{\ast}}g\geq G_+(x,t+t_R)$ and then
\begin{align*}
\|e^{tL_*}g\|_{L^2_{d\mu}}^2
&\geq \int_{\Omega}G_+(x,t+t_R)^2a(x)\,dx
\\
&=
\int_{\R^N\setminus B(0,(t+t_R)^{\frac{1}{2-\alpha}})}G_+(x,t+t_R)^2|x|^{-\alpha}\,dx
\\
&=
(t+t_R)^{-\frac{N-\alpha}{2-\alpha}}\int_{\R^N\setminus B(0,1)}(1-|y|^{2-N})^2\exp\Big[-\frac{|y|^{2-\alpha}}{(2-\alpha)^2}\Big]\,dy
\\
&=
\widetilde{C}(t+t_R)^{-\frac{N-\alpha}{2-\alpha}}.
\end{align*}
To conclude the proof, 
we choose a cut-off function $\widetilde{\eta}(x)\in C_c^\infty(\R^N)$ such that 
\[
\|(1-\widetilde{\eta})g\|_{L^1_{d\mu}}\leq 3^{-\frac{N-\alpha}{2(2-\alpha)}}C^{-1}\widetilde{C},
\]
where $C$ is the constant given by $L^2_{d\mu}$-$L^1_{d\mu}$ estimate. 
Then 
\begin{align*}
\|e^{tL_*}(1-\widetilde{\eta})g\|_{L^2_{d\mu}}
\leq 
Ct^{-\frac{N-\alpha}{2(2-\alpha)}}\|(1-\widetilde{\eta})g\|_{L^1_{d\mu}}
\leq 
\widetilde{C}(3t)^{-\frac{N-\alpha}{2(2-\alpha)}}
\end{align*}
and therefore for every $t\geq t_R$, 
\begin{align*}
\|e^{tL_*}(\widetilde{\eta}g)\|_{L^2_{d\mu}}
&\geq 
\|e^{tL_*}g\|_{L^2_{d\mu}}-
\|e^{tL_*}(1-\widetilde{\eta})g\|_{L^2_{d\mu}}
\\
&\geq 
\widetilde{C}(t+t_R)^{-\frac{N-\alpha}{2(2-\alpha)}}
-\widetilde{C}(3t)^{-\frac{N-\alpha}{2(2-\alpha)}}
\\
&\geq 
\widetilde{C}\left(
2^{-\frac{N-\alpha}{2(2-\alpha)}}
-
3^{-\frac{N-\alpha}{2(2-\alpha)}}
\right)
t^{-\frac{N-\alpha}{2(2-\alpha)}}.
\end{align*}
Since $\widetilde{\eta}g\in C^\infty(\Omega)$ is compactly supported, the proof is completed if $N\geq 3$.

If $N=2$, then replacing $G$ with 
\[
\widetilde{G}(x,t)
=t^{-1}\log\Big(\frac{|x|^{2-\alpha}}{t}\Big)
       \exp\Big[-\frac{|x|^{2-\alpha}}{(2-\alpha)^2t}\Big], 
\quad x\in \Omega, t> 0,
\]
we can verify the same conclusion. Hence the proof is completed. 
\end{proof}

%
%
%

\section{Weighted energy estimates for damped wave equation}
In this section, we prove almost sharp estimates for
time-derivatives of the solution to the damped wave equation \eqref{dw}.

In this section, we assume the coefficient of the damping 
$a(x)$
is radially symmetric function on the whole space $\mathbb{R}^N$
and satisfies
$a\in C^2(\mathbb{R}^N)$
and
\begin{align}
\label{a0p}
\tag{a0$'$}
	a_1 (1+|x|)^{-\alpha} \le a(x) \le a_2 (1+|x|)^{-\alpha}
\end{align}
with some
$a_1, a_2 >0$
and
$\alpha \in [0,1)$.
Then, by Todorova and Yordanov \cite[Proposition 1.3]{ToYo09},
there exists a solution to the Poisson equation
\begin{align}
\label{poi}
	\Delta A_0 (x) = a(x)\quad \mbox{in}\ \mathbb{R}^N
\end{align}
satisfying
\begin{align}
\tag{a1}\label{a1}
	&A_1 (1+|x|)^{2-\alpha} \le A_0(x) \le A_2 (1+|x|)^{2-\alpha},\\
\tag{a2}\label{a2}
	&
	h_a:=
	\lim_{R\to \infty}
	\left[
  		\sup_{x\in \R^N\setminus B_R}\left(\frac{|\nabla A_0(x)|^2}{a(x)A_0(x)}\right)
	\right]<\infty
\end{align}
with some $A_1, A_2 >0$.

\begin{lemma}
\label{lem_a}
Let $a$ and $A_0$ satisfy \eqref{a0p}, \eqref{poi}, \eqref{a1} and \eqref{a2}. 
Then for every $\ep>0$ there exists $c_0 \geq 0$ such that 
$A(x)=A_0(x) + c_0 $ satisfies
\[
	\frac{|\nabla A(x)|^2}{a(x)A(x)} \leq  h_{\alpha}+\ep
\]
for any
$x\in \Omega$.
\end{lemma}
\begin{proof}
For any $\ep > 0$,
there exists a constant $R>0$ such that
\[
	\frac{|\nabla A_0(x)|^2}{a(x)A_0(x)} \leq  h_{\alpha}+\ep
\]
holds for any
$|x| > R$.
For $x\in \Omega \cap \{ |x| \le R \}$,
taking a constant
$c_0>0$
sufficient large, we have
\[
	\frac{|\nabla (A_0(x)+c_0)|^2}{a(x)(A_0(x)+c_0)} 
	\leq 
	\frac{\|\nabla A_0\|_{L^\infty(\Omega \cap \{ |x| \le R \})}^2}{\inf_{x\in \Omega \cap \{ |x| \le R \}}a(x)}c_0^{-1}
	\leq  h_{\alpha}+\ep.
\]
The proof is completed.
\end{proof}

Let us recall the finite speed propagation property of the wave equation
(see \cite{Ikbook}).
\begin{lemma}[Finite speed of propagation]
\label{fp}
Let $u$ be the solution of \eqref{dw} with the initial data
$(u_0,u_1)$
satisfying
${\rm supp}\,(u_0,u_1) \subset \{ x \in \Omega ; |x| \le R_0 \}$.
Then, one has
\[
	{\rm supp}\,u(\cdot,t)\subset \{x\in \Omega\;;\;|x|\leq R_0+t\}
\]
and therefore
$|x|/(t_0+t) \leq 1$
for $t_0\geq R_0, t \ge 0$
and
$x\in {\rm supp}\,u(\cdot, t)$.
\end{lemma}

We define the weight function
$\Phi$.
For a constant
$\beta > 0$
and a function
$A(x)$,
we put
\begin{align*}
	\Phi_{A,\beta}(x,t) &:= \exp\left(\beta\,\frac{A(x)}{1+t}\right).
\end{align*}

The main result of this section is the following.

\begin{proposition}\label{prop_en}
Let
$a(x) \in C^{2}(\mathbb{R}^N)$
be a radially symmetric function
satisfying \eqref{a0p} and let
$A(x)$
be a solution to the Poisson equation \eqref{poi} which satisfies \eqref{a1}, \eqref{a2}.
Then, for any $\ep>0$, there exists a constant
$\beta >0$ such that the following holds:
Let
$k\ge 0$
be an integer and assume that
$a\in C^{\max(k,2)}(\mathbb{R}^N)$
and the initial data
$(u_0,u_1)$
belong to
$[H^{k+1}\cap H^1_0(\Omega)] \times [H^{k}\times H^1_0(\Omega)]$
with the compatibility condition of order
$k$
and the support condition
${\rm supp\,}(u_0,u_1) \subset \{ x \in \Omega ; |x| < R_0 \}$
for some $R_0>0$.
Then, there exist constants
$t_0>0$
and
$C=C(N,k,R_0,\ep )>0$
such that the solution
$u$
to \eqref{dw} satisfies
\begin{align}
\label{en_es}
	&\int_{\Omega} \Phi_{A,\beta}(x,t) a(x) | \pa_t^k u(x,t) |^2 dx \\
\nonumber
	&\quad \le C (t_0+ t)^{-1/h_a-2k+\ep} \| (u_0, u_1) \|_{H^{k+1} \times H^k}^2,\\
\nonumber
	&\int_{\Omega} \Phi_{A,\beta}(x,t) | \pa_t^k \nabla_x u(x,t) |^2 dx \\
\nonumber
	&\quad \le C (t_0+ t)^{-1/h_a-2k-1+\ep} \| (u_0, u_1) \|_{H^{k+1} \times H^k}^2
\end{align}
for
$t\ge 0$.
\end{proposition}

In particular, as stated in Introduction,
if $a(x)$ is a radially symmetric function on the whole space
$\R^N$
and satisfies \eqref{a0} with some
$\alpha \in [0,1)$,
then there exists a solution $A(x)$ to the Poisson equation \eqref{poi}
satisfying
\begin{align}
\tag{a3}\label{a3}
	A(x) &= \frac{a_0}{(2-\alpha)(N-\alpha)}|x|^{2-\alpha} + o(|x|^{2-\alpha}),\\
\tag{a4}\label{a4}
	h_a &= \frac{2-\alpha}{N-\alpha}.
\end{align}
This was also proved in \cite[Proposition 1.3]{ToYo09}.
Therefore, from the above proposition, we obtain the following
energy estimates for time-derivatives of the solution to \eqref{dw}.

\begin{corollary}\label{cor32}
Let
$k \ge 0$
be an integer.
In addition to the assumption in Proposition \ref{prop_en},
we assume that
$a(x)$
satisfies \eqref{a0}.
Then for any
$\ep >0$,
there exist constants
$t_0>0$
and
$C=C(N,k,R,\ep)>0$
such that the solution $u$ to \eqref{dw} satisfies
\begin{align*}
	\int_{\Omega} \Phi_{A,\beta}(x,t)a(x) | \pa_t^k u(x,t) |^2 dx
	&\le C (t_0+ t)^{-\frac{N-\alpha}{2-\alpha} -2k +\ep } \| (u_0, u_1) \|_{H^{k+1} \times H^k}^2,\\
	\int_{\Omega} \Phi_{A,\beta}(x,t) | \pa_t^k \nabla_x u(x,t) |^2 dx
	&\le C (t_0+ t)^{-\frac{N-\alpha}{2-\alpha} -2k -1 +\ep } \| (u_0, u_1) \|_{H^{k+1} \times H^k}^2
\end{align*}
for
$t\ge 0$.
\end{corollary}

In order to prove Proposition \ref{prop_en},
we use a weighted energy method which originates from
Todorova and Yordanov \cite{ToYo01}, \cite{ToYo09} and Ikehata \cite{Ik05IJPAM}.
Our proof is based on the argument of \cite{Ni10} and a Hardy-type inequality
Lemma \ref{lem_ha}.

\begin{lemma}
\label{fundamental}
We have
\begin{align*}
	\pa_t \Phi_{A,\beta}
	&= - \beta \frac{A}{(1+t)^2}\Phi_{A,\beta}, \\
	\nabla\Phi_{A,\beta}
	&= \beta \frac{\nabla A}{1+t}\Phi_{A,\beta},\\
	\Delta\Phi_{A,\beta}
	&= \beta \frac{\Delta A}{1+t}\Phi_{A,\beta}
		+  \left|\beta\frac{\nabla A}{1+t}\right|^2 \Phi_{A,\beta}.
\end{align*}
In particular, we have
\begin{align*}
	-\Delta\Phi_{A,\beta}+\frac{|\nabla \Phi_{A,\beta}|^2}{\Phi_{A,\beta}}
	&= -\beta \frac{a(x)}{1+t}\Phi_{A,\beta},\\
	\frac{|\nabla \Phi_{A,\beta}|^2}{\pa_t \Phi_{A,\beta}}
	&=-\beta \frac{|\nabla A|^2}{A}\Phi_{A,\beta}.
\end{align*}
\end{lemma}

\begin{proof}
It suffices to observe that
\begin{align*}
\pa_t \Phi_{A,\beta}(x,t)
&=
\exp\left[\beta \frac{A(x)}{1+t}\right]
\left(
-\beta \frac{A(x)}{(1+t)^2}
\right)
=
-\beta \frac{A(x)}{(1+t)^2}
\Phi_{A,\beta}(x,t), 
\\[10pt]
\nabla\Phi_{A,\beta}(x,t)
&=
\exp\left[\beta \frac{A(x)}{1+t}\right]
\left(
\beta \frac{\nabla A}{1+t}
\right)
=
 \beta\frac{\nabla A}{1+t}
\Phi_{A,\beta}(x,t),
\\[10pt]
\Delta\Phi_{A,\beta}(x,t)
&=
\beta \frac{\Delta A}{1+t}\Phi_{A,\beta}(x,t)
+ 
\left|\beta\frac{\nabla A}{1+t}\right|^2
\Phi_{A,\beta}(x,t).
\end{align*}
Combining the equalities above, we deduce the assertion.
\end{proof}

Next, we prove a Hardy-type inequality with the weight function
$\Phi_{A,\beta}$,
which is essential in our weighted energy method.
\begin{lemma}
\label{lem_ha}
We have
\begin{align}
\label{hardy}
	\frac{\beta}{1+t}\int_{\Omega} a(x)|u|^2 \Phi_{A,\beta}\,dx
	\leq  \int_{\Omega} |\nabla u|^2 \Phi_{A,\beta}\,dx.
\end{align}
\end{lemma}

\begin{proof}
Integration by parts gives
\begin{align*}
0&\leq\int_{\Omega}
   \Phi_{A,\beta}^{-1}
   |\nabla (\Phi_{A,\beta} u)|^2
\,dx
\\
&=
\int_{\Omega}
   \Big(
      \Phi_{A,\beta}|\nabla u|^2
      +2(\nabla\Phi_{A,\beta}\cdot \nabla u)u
      +\frac{|\nabla \Phi_{A,\beta}|^2}{\Phi_{A,\beta}}|u|^2
   \Big)
\,dx
\\
&=
\int_{\Omega}
      \Phi_{A,\beta}|\nabla u|^2
\,dx
+
\int_{\Omega}
   \Big(
      -\Delta \Phi_{A,\beta}
      +\frac{|\nabla \Phi_{A,\beta}|^2}{\Phi_{A,\beta}}
   \Big)|u|^2
\,dx.
\end{align*}
By Lemma \ref{fundamental},
we have
\[
	-\Delta \Phi_{A,\beta} + \frac{|\nabla \Phi_{A,\beta}|^2}{\Phi_{A,\beta}}
	= -\beta \frac{\Delta A(x)}{1+t} \Phi_{A,\beta}
	= -\beta \frac{a(x)}{1+t} \Phi_{A,\beta}.
\]
Thus, we obtain \eqref{hardy}.
\end{proof}


Let us turn to the proof of Proposition \ref{prop_en}.
In what follows, we fix an arbitrary constant
$\ep > 0$
and take a constant
$c_0>0$ so that Lemma \ref{lem_a} holds.

Let us define
\begin{align*}
	E_1(t;u)&= \int_{\Omega} (|\nabla u|^2 + |u_t|^2) \Phi_{A,\beta} dx, \\
	E_2(t;u)&= \int_{\Omega} (2uu_t + |u|^2) \Phi_{A,\beta} dx, \\
	F(t;u) &= \int_{\Omega} \left( a(x) + \frac{A(x)}{(1+t)^2} \right) |u_t|^2 \Phi_{A,\beta} dx.
\end{align*}

\begin{remark}
In order to prove Proposition \ref{prop_en} for
$k\in \mathbb{Z}_{\ge 0}$,
we assume additional regularity on the initial data:
\[
	(u_0,u_1)\in
	[H^{k+2}(\Omega)\cap H^1_0(\Omega)]\times [H^{k+1}(\Omega)\cap H^1_0(\Omega)]
\]
and hence, the solution
$u$
of \eqref{dw} has the regularity
$\cap_{j=0}^{k+2}C^{j}([0,\infty);H^{k+2-j}(\Omega))$,
since the initial data satisfies the compatibility condition of order $k$
(see \cite[Theorem 2]{Ik68}).
After proving the conclusion of Proposition \ref{prop_en} for sufficiently regular initial data,
we apply the density argument and obtain the same estimate
for the initial data belonging to
$[H^{k+1}(\Omega)\cap H^1_0(\Omega)]\times [H^{k}(\Omega)\cap H^1_0(\Omega)]$.
\end{remark}

\begin{lemma}
\label{lem_en0}
Let $\Phi\in C^2(\overline{\Omega}\times [0,\infty))$ 
satisfy $\Phi>0$ and $\pa_t\Phi<0$
and let $u$ be a solution of \eqref{dw}. Then
\begin{align*}
	\frac{d}{dt}\left[ \int_{\Omega} \Big(|\nabla u|^2+|u_t|^2\Big) \Phi\,dx \right]
	&= \int_{\Omega} (\pa_t\Phi)^{-1} \big| \pa_t\Phi\nabla u -u_t\nabla\Phi \big|^2 \,dx\\
	&\quad + \int_{\Omega}
		\Big( -2a(x)\Phi+\pa_t\Phi - (\pa_t\Phi)^{-1}|\nabla\Phi|^2 \Big)|u_t|^2 \,dx.   
\end{align*}
In particular,
if we put
$\Phi=\Phi_{A,\beta}$,
and
$\beta < 2 (h_a+\ep)^{-1}$,
then
there exists
$\gamma_1(\beta) >0$
such that 
\begin{align}
\label{eq_en1}
	&\frac{d}{dt} E_1(t;u)
	\leq  -\gamma_1(\beta) F(t;u).
\end{align}
\end{lemma}

\begin{proof}
We have
\begin{align*}
	&\frac{d}{dt}\left[ \int_{\Omega} \Big(|\nabla u|^2+|u_t|^2\Big) \Phi\,dx \right]\\
	&= \int_{\Omega} |\nabla u|^2 (\pa_t\Phi)\,dx  
	+ 2\int_{\Omega} (\nabla u\cdot \nabla u_t) \Phi\,dx  
	+2\int_{\Omega} u_{tt}u_t \Phi\,dx  
	+\int_{\Omega} |u_t|^2 (\pa_t\Phi)\,dx\\
	&= \int_{\Omega} |\nabla u|^2 (\pa_t\Phi)\,dx  
	+ 2\int_{\Omega} (u_{tt}-\Delta u)u_t \Phi\,dx \\
	&\quad - 2\int_{\Omega} (\nabla u\cdot \nabla \Phi)u_t\,dx
	+\int_{\Omega} |u_t|^2 (\pa_t\Phi)\,dx. 
\end{align*}
Since $u$ satisfies \eqref{dw}, we see that
$u_{tt}-\Delta u = -a(x)u_t$.
Moreover, we have
\[
	(\partial_t\Phi) |\nabla u|^2 -2(\nabla u\cdot \nabla \Phi)u_t
	= (\partial_t\Phi)^{-1}| (\partial_t\Phi)\nabla u - u_t\nabla\Phi |^2
		- (\partial_t \Phi)^{-1} |\nabla \Phi|^2 |u_t|^2.
\]
Thus, we have the first assertion.
In particular, if
$\Phi = \Phi_{A,\beta}$,
then
\[
	-2a(x)\Phi+\pa_t\Phi - (\pa_t\Phi)^{-1}|\nabla\Phi|^2
	\le - \left( 2 a(x) +\beta \frac{A}{(1+t)^2}
		- \beta \frac{|\nabla A|^2}{A} \right) \Phi_{A,\beta}.
\]
By Lemma \ref{lem_a},
$|\nabla A(x)|^2/A(x) \le (h_a+\ep)a(x)$.
This and
$\beta < 2(h_a+\ep)^{-1}$
lead to
$-2a(x)+\beta |\nabla A(x)|^2/A(x) \le -(2-\beta(h_a+\ep))a(x)$.
Therefore, there exists some constant $\gamma_1(\beta)>0$ such that
\begin{align*}
	\int_{\Omega}
		\Big( -2a(x)\Phi_{A,\beta}+\pa_t\Phi_{A,\beta}
			- (\pa_t\Phi_{A,\beta})^{-1}|\nabla\Phi_{A,\beta}|^2 \Big)|u_t|^2 \,dx
	\le -\gamma_1(\beta) F(t;u).
\end{align*}
From this and the first assertion, we reach the desired estimate.
\end{proof}

Furthermore, by multiplying
\eqref{eq_en1} by
$(t_0+t)^{m}$
with arbitrary $m\geq 0$,
we have the following.
\begin{lemma}
\label{lem_en1}
Let $u$ be a solution of \eqref{dw}
and let $m\geq 0$.
Then there exists $t_1\geq 1$ such that 
for every $t_0\geq t_1$ and $t>0$,
\begin{align*}
	\frac{d}{dt} \left[ (t_0+t)^{m} E_1(t;u) \right]
	\leq 
	m (t_0+t)^{m-1} \int_{\Omega} |\nabla u|^2 \Phi_{A,\beta}\,dx
	- \frac{\gamma_1(\beta)}{2} (t_0+t)^m F(t;u).
\end{align*}
\end{lemma}

\begin{proof}
Let $t_1$ be a constant determined later and let $t_0 \ge t_1$.
Multiplying \eqref{eq_en1} by
$(t_0+t)^{m}$,
we have
\begin{align*}
	&\frac{d}{dt}
	\left[ (t_0+t)^{m} E_1(t;u)  \right] \\
	&\leq m (t_0+t)^{m-1} E_1(t;u) -\gamma_1(\beta)(t_0+t)^{m}F(t;u) \\
	&\leq  m (t_0+t)^{m-1} \int_{\Omega} |\nabla u|^2 \Phi_{A,\beta}\,dx \\
	&\quad + (t_0+t)^m
		\left[ m (t_0+t)^{-1} \int_{\Omega} |u_t|^2\Phi_{A,\beta}\,dx -\gamma_1(\beta) F(t;u) \right].
\end{align*}
By Lemma \ref{fp}, we have
$a(x) \ge a_1 (1+|x|)^{-\alpha} \ge a_1 t_1^{-\alpha}$
if $t_1 \ge R_0+1$.
Since $\alpha < 1$, retaking $t_1$ larger so that
\[
	m (t_0 + t)^{-1}
	\le a_1^{-1} m (t_0 + t)^{\alpha-1} a(x)
	\le \frac{\gamma_1(\beta)}{2} a(x),
\]
we have the desired estimate.
\end{proof}

\begin{lemma}
Let $\Phi\in C^2(\overline{\Omega}\times [0,\infty))$ 
satisfy $\Phi>0$ and $\pa_t\Phi<0$
and let $u$ be a solution to \eqref{dw}.
Then, we have
\begin{align*}
	&\frac{d}{dt} \left[ \int_{\Omega} \Big(2uu_t+a(x)|u|^2\Big) \Phi\,dx  \right] \\
	&= 2\int_{\Omega} uu_{t} (\pa_t\Phi)\,dx
	+ 2\int_{\Omega} |u_t|^2 \Phi\,dx  
	- 2\int_{\Omega} |\nabla u|^2 \Phi\,dx
	+ \int_{\Omega} \big(a(x)\pa_t\Phi+\Delta\Phi\big)|u|^2\,dx.
\end{align*}
In particular, if
$\Phi = \Phi_{A,\beta}$
and
$0<\beta< (h_a+\ep)^{-1}$,
then there exists $\gamma_2(\beta) > 0$ such that
\begin{align}
\label{eq_en2}
	&\frac{d}{dt} E_2(t;u)+
	\int_{\Omega} |\nabla u|^2 \Phi_{A,\beta}\,dx
	\leq \gamma_2(\beta) (R_0 + 1+ t)^{\alpha}F(t;u).
\end{align}
\end{lemma}
\begin{proof}
Since $u$ satisfies \eqref{dw}, we see that
\begin{align*}
	&\frac{d}{dt} \left[
	\int_{\Omega} \Big(2uu_t+a(x)|u|^2\Big) \Phi\,dx  \right] \\
	&= 2\int_{\Omega} uu_{t} (\pa_t\Phi)\,dx
	+ \int_{\Omega} a(x)|u|^2 (\pa_t\Phi)\,dx \\
	&\quad
	+ 2\int_{\Omega} |u_t|^2 \Phi\,dx  
	+ 2\int_{\Omega} u(u_{tt}+a(x)u_t) \Phi\,dx \\
	&= 2\int_{\Omega} uu_{t} (\pa_t\Phi)\,dx   
	+ \int_{\Omega} a(x)|u|^2 (\pa_t\Phi)\,dx
	+ 2\int_{\Omega} |u_t|^2 \Phi\,dx  
	+ 2\int_{\Omega} u(\Delta u) \Phi\,dx.
\end{align*}
Noting that integration by parts yields 
\begin{align*}
	\int_{\Omega} u(\Delta u) \Phi\,dx
	&= -\int_{\Omega} |\nabla u|^2 \Phi\,dx
	- \int_{\Omega} u(\nabla u\cdot\nabla\Phi)\,dx \\
	&= -\int_{\Omega} |\nabla u|^2 \Phi\,dx
	- \frac{1}{2}\int_{\Omega} \big(\nabla(|u|^2)\cdot\nabla\Phi\big)\,dx \\
	& = -\int_{\Omega} |\nabla u|^2 \Phi\,dx
	+\frac{1}{2} \int_{\Omega} (\Delta\Phi)|u|^2\,dx,
\end{align*}
we deduce
\begin{align*}
&\frac{d}{dt}
\left[
\int_{\Omega}
   \Big(2uu_t+a(x)|u|^2\Big)
\Phi\,dx 
\right]
+
2\int_{\Omega}
   |\nabla u|^2
\Phi\,dx 
\\
&=
2\int_{\Omega}
   uu_{t}
(\pa_t\Phi)\,dx
+
2\int_{\Omega}
   |u_t|^2
\Phi\,dx   
+
\int_{\Omega}
   \big(a(x)\pa_t\Phi+\Delta\Phi\big)|u|^2\,dx.
\end{align*}
This proves the first assertion.
If $\Phi=\Phi_{A,\beta}$, then from Lemma \ref{fundamental} we obtain
\[
	a(x)\pa_t\Phi_{A,\beta}+\Delta \Phi_{A,\beta}
	= - \beta \frac{a(x)A(x)}{(1+t)^2} \Phi_{A,\beta}
		+ \beta \frac{a(x) }{1+t} \Phi_{A,\beta}
		+ \left| \beta \frac{\nabla A(x)}{1+t} \right|^2 \Phi_{A,\beta}.
\]
By Lemma \ref{lem_a}, we have
$|\nabla A(x)|^2 \le (h_a+\ep)a(x)A(x)$
and hence,
\[
	a(x)\pa_t\Phi_{A,\beta}+\Delta \Phi_{A,\beta}
      \leq
      \frac{\beta}{1+t}a(x)\Phi_{A,\beta}
	- \frac{\beta(1-\beta(h_a+\ep))}{(1+t)^2}a(x)A(x)\Phi_{A,\beta}.
\]
From this and Lemma \ref{lem_ha}, we conclude
\begin{align*}
	&\int_{\Omega} \big(a(x)\pa_t\Phi_{A,\beta}+\Delta\Phi_{A,\beta} \big)|u|^2\,dx\\
	&\quad \le \int_{\Omega} |\nabla u|^2 \Phi_{A,\beta}\,dx
		- \frac{\beta(1-\beta(h_a+\ep))}{(1+t)^2}
			\int_{\Omega} a(x)A(x) |u|^2 \Phi_{A,\beta}\,dx.
\end{align*}
On the other hand, the Young inequality and Lemma \ref{fp} yield
\begin{align*}
2\int_{\Omega}
   uu_{t}
(\pa_t\Phi_{A,\beta})\,dx
&=
\frac{2\beta}{(1+t)^2}\int_{\Omega}
   uu_{t}
A\Phi_{A,\beta}\,dx
\\
&\leq 
\frac{4\beta}{(1-\beta(h_a+\ep))(1+t)^2}\int_{\Omega}
a(x)^{-1}A(x)|u_{t}|^2\Phi_{A,\beta}\,dx
\\
&\quad+
\frac{\beta(1-\beta(h_a+\ep))}{(1+t)^2}\int_{\Omega}
a(x)A(x)|u|^2\Phi_{A,\beta}\,dx
\\
&\leq 
\frac{4\beta(R_0+1+t)^{\alpha}}{c_0(1-\beta(h_a+\ep))(1+t)^2}\int_{\Omega}
A(x)|u_{t}|^2\Phi_{A,\beta}\,dx
\\
&\quad+
\frac{\beta(1-\beta(h_a+\ep))}{(1+t)^2}\int_{\Omega}
a(x)A(x)|u|^2\Phi_{A,\beta}\,dx.
\end{align*}
Therefore, we have
\begin{align*}
	&\frac{d}{dt}E_2(t;u)
	+ \int_{\Omega} |\nabla u|^2 \Phi_{A,\beta}\,dx \\
	&\leq  2\int_{\Omega} |u_t|^2 \Phi_{A,\beta}\,dx
	+ \frac{4\beta(R_0+1+t)^{\alpha}}{(1-\beta(h_a+\ep))(1+t)^2}
	\int_{\Omega} A(x) |u_t|^2 \Phi_{A,\beta}\,dx.
\end{align*}
Finally, Lemma \ref{lem_a} gives
\[
	2\int_{\Omega} |u_t|^2 \Phi_{A,\beta}\,dx
	\le 2a_1^{-1} (R_0+1+t)^{\alpha} \int_{\Omega} a(x) |u_t|^2 \Phi_{A,\beta}\, dx
\]
and hence,
\[
	\frac{d}{dt}E_2(t;u)
	+ \int_{\Omega} |\nabla u|^2 \Phi_{A,\beta}\,dx \\
	\leq \gamma_2(\beta) (R_0+1+t)^{\alpha} F(t;u)
\]
with some $\gamma_2(\beta) >0$,
which completes the proof.
\end{proof}

As in Lemma \ref{lem_en1} by multiplying
\eqref{eq_en2} by
$(t_0+t)^{l}$
with arbitrary $l\geq 0$,
we have the following.

\begin{lemma}
\label{lem_en2}
Let
$l\geq 0$.
There exists $t_2\geq t_1$ such that for every
$t_0 \ge t_2$
and
$t\geq 0$, 
\begin{align*}
	&\frac{d}{dt} \left[ (t_0+t)^{l} E_2(t;u) \right]
		+ (t_0+t)^l \int_{\Omega} |\nabla u|^2 \Phi_{A,\beta}\,dx \\
	&\quad \leq
		l (1+\beta \ep ) (t_0+t)^{l-1} \int_{\Omega} a(x) |u|^2 \Phi_{A,\beta}\,dx
	+ 2\gamma_2(\beta)(t_0+t)^{l+\alpha}F(t;u).
\end{align*}
\end{lemma}
\begin{proof}
Let $l \ge 0$ and let $t_2 \ge t_1$ be a constant determined later.
We assume that
$t_0 \ge t_2$.
Multiplying \eqref{eq_en2} by
$(t_0+t)^{l}$,
we have
\begin{align*}
	&\frac{d}{dt}
	\left[
	(t_0+t)^{l}
	E_2(t;u)
	\right]
	+
	(t_0+t)^{l}
	\int_{\Omega}
	|\nabla u|^2
	\Phi_{A,\beta}\,dx
	\\
	&\quad \leq 
	(t_0+t)^{l}
	\left[
	l (t_0+t)^{-1}
	E_2(t;u)
	+
	\gamma_2(\beta) (R_0+1+t)^{\alpha} F(t; u)
	\right].
\end{align*}
We estimate the right-hand side.
From the Schwarz inequality and 
Lemma \ref{fp}, we obtain
\begin{align*}
	&2l (t_0+t)^{-1}
	\int_{\Omega} uu_t\Phi_{A,\beta}\,dx \\
	&\leq 
		\frac{l^2}{\gamma_2(\beta)} (t_0+t)^{-2+\alpha}
	\int_{\Omega} a(x)|u|^2 \Phi_{A,\beta}\,dx
	+ \gamma_2(\beta) (t_0+t)^{-\alpha} \int_{\Omega} a(x)^{-1}|u_t|^2 \Phi_{A,\beta}\,dx \\
	&\leq
		\frac{l^2}{\gamma_2(\beta)} (t_0+t)^{-2+\alpha}
	\int_{\Omega} a(x)|u|^2 \Phi_{A,\beta}\,dx
	+ \gamma_2(\beta) (t_0+t)^{\alpha} \int_{\Omega} a(x) |u_t|^2 \Phi_{A,\beta}\,dx \\
	&\leq
	\frac{l^2}{\gamma_2(\beta)}(t_0+t)^{-2+\alpha}
	\int_{\Omega} a(x)|u|^2 \Phi_{A,\beta}\,dx
	+ \gamma_2(\beta) (t_0+t)^{\alpha} F(t;u).
\end{align*}
Here we have also used Lemma \ref{fp}
and the definition of
$F(t; u)$.
Taking $t_2 \ge t_1$ such that 
\[
	\frac{l^2}{\gamma_2(\beta)} t_2^{-1+\alpha}
	\leq l \beta \ep 
\]
and hence,
\begin{align*}
	\frac{l^2}{\gamma_2(\beta)}(t_0+t)^{-2+\alpha}
	\int_{\Omega} a(x) |u|^2 \Phi_{A,\beta}\,dx
	&\le
	l \beta \ep (t_0+t)^{-1} \int_{\Omega} a(x) |u|^2 \Phi_{A,\beta}\,dx.
\end{align*}
Consequently, we reach
\begin{align*}
	&\frac{d}{dt} \left[ (t_0+t)^{l} E_2(t;u) \right]
		+ (t_0+t)^l \int_{\Omega} |\nabla u|^2 \Phi_{A,\beta}\,dx \\
	&\quad \leq
		l (1+\beta \ep ) (t_0+t)^{l-1} \int_{\Omega} a(x) |u|^2 \Phi_{A,\beta}\,dx
	+ 2\gamma_2(\beta)(t_0+t)^{l+\alpha}F(t;u).
\end{align*}
This completes the proof.
\end{proof}

In particular,
by choosing
$l = (h_a+4\ep)^{-1}$
and
$\beta=(h_a+2\ep)^{-1}$,
we have the following estimate.
\begin{lemma}
\label{lem_en3}
For every $t_0\ge t_2$ and $t\geq 0$, we have
\begin{align*}
	&\frac{d}{dt} \left[ (t_0+t)^{\frac{1}{h_a+4\ep}} E_2(t;u) \right] \\
	&\leq - \frac{\ep}{h_a+4\ep}(t_0+t)^{\frac{1}{h_a+4\ep}}
	\int_{\Omega} |\nabla u|^2 \Phi_{A,\beta}\,dx
	+2\gamma_2(\beta)(t_0+t)^{\frac{1}{h_a+4\ep}}F(t;u).
\end{align*}
\end{lemma}

\begin{proof}
We choose
$l = (h_a+4\ep)^{-1}$
and
$\beta=(h_a+2\ep)^{-1}$
in Lemma \ref{lem_en2}.
Moreover, Lemma \ref{lem_ha} implies
\[
	l(1+\beta \ep)(t_0+t)^{l-1} \int_{\Omega} a(x) |u|^2 \Phi_{A,\beta}\,dx
	\le \frac{h_a+3\ep}{h_a+4\ep} (t_0+t)^{l} \int_{\Omega} |\nabla u|^2 \Phi_{A,\beta}\, dx,
\]
which gives the desired estimate.
\end{proof}

From Lemmas \ref{lem_en1} and \ref{lem_en3},
we obtain Proposition \ref{prop_en} when
$k=0$.
\begin{proof}[Proof of Proposition \ref{prop_en} when $k=0$]
Let
$t_0 \ge t_2$
be a constant determined later.
Taking
$m=(h_a+4\ep)^{-1}+\alpha$
in Lemma \ref{lem_en1}, we have
\begin{align*}
	\frac{d}{dt}\left[ (t_0+t)^{\frac{1}{h_a+4\ep}+\alpha} E_1(t;u) \right]
	&\le \left( \frac{1}{h_a+4\ep}+\alpha \right) (t_0+t)^{\frac{1}{h_a+4\ep}+\alpha-1}
		\int_{\Omega} |\nabla u|^2 \Phi_{A,\beta}\,dx \\
	&\quad - \frac{\gamma_1(\beta)}{2} (t_0+t)^{\frac{1}{h_a+4\ep}+\alpha} F(t;u).
\end{align*}
Set
$\nu= \min\{ \frac{\gamma_1}{4\gamma_2}, \frac{1}{4a_2}\}$
and define
\[
	\mathcal{E}_0(t;u) = (t_0+t)^{\frac{1}{h_a+4\ep}+\alpha} E_1(t;u)
		+ \nu (t_0+t)^{\frac{1}{h_a+4\ep}} E_2(t;u).
\]
Then, the Schwarz inequality and Lemma \ref{fp} imply
\begin{align*}
	2 \nu |u u_t|
	&\le \frac{1}{2}a(x) |u|^2 + 2\nu a(x)^{-1} |u_t|^2 \\
	&\le \frac{1}{2} a(x) |u|^2 + \frac{\nu a_2}{2} (t_0+t)^{\alpha} |u_t|^2 \\
	&\le \frac{1}{2} a(x) |u|^2 + \frac{1}{2} (t_0+t)^{\alpha} |u_t|^2
\end{align*}
and hence,
$\mathcal{E}(t;u)$ is equivalent to
\[
	(t_0+t)^{\frac{1}{h_a+4\ep}+\alpha}
		\int_{\Omega} \Big(|\nabla u|^2+|u_t|^2\Big) \Phi_{A,\beta}\,dx
	+ (t_0+t)^{\frac{1}{h_a+4\ep}} \int_{\Omega} a(x)|u|^2 \Phi_{A,\beta}\,dx.
\]
By Lemmas \ref{lem_en1} and \ref{lem_en2} with $l=(h_a+4\ep)^{-1}$, we have
\begin{align*}
	&\frac{d}{dt} \mathcal{E}_0(t;u) \\
	&\leq  (t_0+t)^{\frac{1}{h_a+4\ep}}
	\left[ \left(\frac{1}{h_a+4\ep}+\alpha\right) (t_0+t)^{\alpha-1}
	- \frac{\ep \nu}{h_a+4\ep} \right]
	\int_{\Omega} |\nabla u|^2 \Phi_{A_\ep,\beta}\,dx. 
\end{align*}
Therefore, there exists $t_3\geq t_2$ such that 
for every $t_0 \ge t_3$ and $t\geq 0$, 
\begin{align*}
	&\frac{d}{dt} \mathcal{E}_0(t;u)
	\leq  -\frac{\ep \nu}{2(h_a+4\ep)} (t_0+t)^{\frac{1}{h_a+4\ep}}
	\int_{\Omega} |\nabla u|^2 \Phi_{A,\beta}\,dx.
\end{align*}
Integrating it over $[0,t]$, we have
\begin{align}
\label{en3}
	\mathcal{E}_0(t;u)
	+ \frac{\ep \nu}{2(h_a+4\ep)} \int_0^t (t_0+s)^{\frac{1}{h_a+4\ep}}
	\int_{\Omega} |\nabla u|^2 \Phi_{A,\beta}\,dx ds
	\le \mathcal{E}_0(0;u).
\end{align}
In particular, we obtain
\[
	(t_0+t)^{\frac{1}{h_a+4\ep}} \int_{\Omega} a(x) |u|^2 \Phi_{A,\beta}\,dx
	\le C \mathcal{E}_0(0,u)
\]
with some constant
$C>0$,
which depending on
$t_0, \nu, \ep$.
This gives the first assertion of \eqref{en_es} when $k=0$.

Next, we prove the second assertion of \eqref{en_es}.
We note that the estimate \eqref{en3} also gives
\begin{align}
\label{en4}
	\int_0^t (t_0+s)^{\frac{1}{h_a+4\ep}}
	\int_{\Omega} |\nabla u|^2 \Phi_{A,\beta}\,dx ds
	\le C \mathcal{E}_0(0,u).
\end{align}
By choosing
$m= (h_a+4\ep)^{-1}+1$
in Lemma \ref{lem_en1}, we have
\begin{align*}
	&\frac{d}{dt} \left[ (t_0+t)^{\frac{1}{h_a+4\ep}+1} E_1(t;u) \right] \\
	&\le \left( \frac{1}{h_a+4\ep}+1 \right) (t_0+t)^{\frac{1}{h_a+4\ep}}
		\int_{\Omega} |\nabla u|^2 \Phi_{A,\beta}\,dx
	-\frac{\gamma_1(\beta)}{2}(t_0+t)^{\frac{1}{h_a+4\ep}+1} F(t;u).
\end{align*}
Integrating it over $[0,t]$ and using \eqref{en4}, we deduce
\begin{align*}
	&(t_0+t)^{\frac{1}{h_a+4\ep}+1} E_1(t;u)
	+ \int_0^t (t_0+s)^{\frac{1}{h_a+4\ep}+1} F(s;u) ds \\
	&\quad \le C E_1(0;u)+ C\int_0^t (t_0+s)^{\frac{1}{h_a+4\ep}}
		\int_{\Omega} |\nabla u|^2 \Phi_{A,\beta}\,dx ds \\
	&\quad \le C \mathcal{E}_0(0;u).
\end{align*}
In particular, we obtain
\begin{align}
\label{est_E_1F}
	&(t_0+t)^{\frac{1}{h_a+4\ep}+1}
	\int_{\Omega} |\nabla u|^2 \Phi_{A,\beta}\,dx
+
	\int_0^t (t_0+s)^{\frac{1}{h_a+4\ep}+1}F(s,u)\,ds
	\le C \mathcal{E}_0(0;u).
\end{align}
This proves the second assertion of \eqref{en_es} when $k=0$.
\end{proof}

Finally, we give a proof of Proposition \ref{prop_en} when $k\ge 1$.
\begin{proof}[Proof of Proposition \ref{prop_en} for $k\ge 1$.]
For $k \in \mathbb{Z}_{\ge 0}$, we define
\[
	\mathcal{E}_k(t;u)
	= (t_0+t)^{\frac{1}{h_a+4\ep}+\alpha+2k} E_1(t;\partial_t^ku)
		+\nu (t_0+t)^{\frac{1}{h_a+4\ep}+2k} E_2(t;\partial_t^ku).
\]
By differentiating the equation \eqref{dw} with respect to $t$,
we notice that
$u_t$
also satisfies the equation \eqref{dw}.
Therefore,
we apply Lemmas \ref{lem_en1} and \ref{lem_en2}
for
$u_t$
with
$m=(h_a+4\ep)^{-1}+2+\alpha$
and
$l=(h_a+4\ep)^{-1}+2$
to obtain
\begin{align*}
	&\frac{d}{dt}\left[ (t_0+t)^{\frac{1}{h_a+4\ep}+\alpha+2}E_1(t;u_t) \right] \\
	&\quad \le \left( \frac{1}{h_a+4\ep}+\alpha+2 \right)
		(t_0+t)^{\frac{1}{h_a+4\ep}+\alpha+1}
		\int_{\Omega} |\nabla u_t|^2 \Phi_{A,\beta}\,dx \\
	&\qquad -\frac{\gamma_1(\beta)}{2} (t_0+t)^{\frac{1}{h_a+4\ep}+\alpha+2} F(t;u_t)
\end{align*}
and
\begin{align*}
	&\frac{d}{dt} \left[ (t_0+t)^{\frac{1}{h_a+4\ep}+2} E_2(t;u_t) \right]
		+ (t_0+t)^{\frac{1}{h_a+4\ep}+2} \int_{\Omega} |\nabla u_t|^2 \Phi_{A,\beta}\,dx\\
	&\quad \le
		\left( \frac{1}{h_a+4\ep}+2 \right)(1+\beta \ep)
			(t_0+t)^{\frac{1}{h_a+4\ep}+1} \int_{\Omega} a(x) |u_t|^2 \Phi_{A,\beta}\,dx\\
	&\qquad + 2\gamma_2(\beta) (t_0+t)^{\frac{1}{h_a+4\ep}+2+\alpha} F(t;u_t).
\end{align*}
Hence, there exists a constant $t_4 \ge t_3$ such that for every $t_0 \ge t_4$ and $t \ge 0$,
\begin{align*}
	& \frac{d}{dt}\mathcal{E}_1(t;u)
	+ \frac{\nu}{2} (t_0+t)^{\frac{1}{h_a+4\ep}+2}
		\int_{\Omega} |\nabla u_t|^2 \Phi_{A,\beta}\,dx \\
	& \le \left( \frac{1}{h_a+4\ep}+2 \right)(1+\beta\ep)
		(t_0+t)^{\frac{1}{h_a+4\ep}+1} \int_{\Omega} a(x) |u_t|^2 \Phi_{A,\beta}\,dx\\
	& \le \left( \frac{1}{h_a+4\ep}+2 \right)(1+\beta\ep)
		(t_0+t)^{\frac{1}{h_a+4\ep}+1} F(t;u).
\end{align*}
Integrating it over $[0,t]$ and \eqref{est_E_1F}
lead to
\begin{align*}
	&\mathcal{E}_1(t;u)
	+ \int_0^t (t_0+s)^{\frac{1}{h_a+4\ep}+2} \int_{\Omega} |\nabla u_t|^2 \Phi_{A,\beta}\,dx \\
	&\le C\int_0^t (t_0+s)^{\frac{1}{h_a+4\ep}+1} F(s;u) ds\\
	& \le C \left( \mathcal{E}_1(0;u)+\mathcal{E}_0(0;u) \right).
\end{align*}
In particular, we have the first assertion of \eqref{en_es} for $k=1$,
since
$\mathcal{E}_1(0;u)+\mathcal{E}_0(0;u) \le C \| (u_0,u_1)\|_{H^2\times H^1}^2$.
Moreover, we obtain
\[
	\int_0^t (t_0+s)^{\frac{1}{h_a+4\ep}+2} \int_{\Omega} |\nabla u_t|^2 \Phi_{A,\beta}\,dx
	\le C\| (u_0,u_1)\|_{H^2\times H^1}^2.
\]
This estimate together with Lemma \ref{lem_en1} with
$m=(h_a+4\ep)^{-1}+3$
imply
\begin{align*}
	&(t_0+t)^{\frac{1}{h_a+4\ep}+3} E_1(t;u_t)
	+\frac{\gamma_2(\beta)}{2}
	\int_0^t (t_0+s)^{\frac{1}{h_a+4\ep}+3}F(s,u_t)\,ds
	\\
	&\le CE_1(0,u_t)
	+ C\int_0^t (t_0+s)^{\frac{1}{h_a+4\ep}+2} \int_{\Omega} |\nabla u|^2 \Phi_{A,\beta}\,dx
\end{align*}
and hence,
\begin{align*}
	&(t_0+t)^{\frac{1}{h_a+4\ep}+3}\int_{\Omega} |\nabla u_t|^2 dx
+\int_0^t (t_0+s)^{\frac{1}{h_a+4\ep}+3}F(s,u_t)\,ds
\le C\| (u_0,u_1)\|_{H^2\times H^1}^2,
\end{align*}
which gives the second assertion of \eqref{en_es} for $k=1$.
We can repeat the same argument inductively and
we obtain \eqref{en_es} for any $k\in \mathbb{Z}_{\ge 0}$.
\end{proof}


\section{Diffusion phenomena for the damped wave equation}
In this section, we give a proof of Theorem \ref{thm1}.
The argument is similar to that of \cite{Wa14}.
We will apply Proposition \ref{prop_en} with $k=2$ and
assume that the initial data satisfies
\[
	(u_0,u_1) \in [H^3(\Omega)\cap H^1_0(\Omega)]
		\times [H^2(\Omega)\cap H^1_0(\Omega)]
\]
with the compatibility condition of second order
and
${\rm supp}\,(u_0,u_1)\subset \{ x\in \Omega ; |x| \le R_0 \}$
with some
$R_0>0$.
Hence, the corresponding solution $u$ of \eqref{dw} has the regularity
$u\in \cap_{j=0}^3C^j([0,\infty);H^{3-j}(\Omega))$
and
${\rm supp}\, u(\cdot,t) \subset \{ x \in \Omega ; |x| \le R_0 + t \}$.
We rewrite the equation \eqref{dw} as the heat equation on $L^2_{d\mu}$ defined in Section 2:
\[
	u_t - a(x)^{-1}\Delta u = -a(x)^{-1}u_{tt}.
\]
Then we remark that
\begin{lemma}\label{Duhamel}
Assume that the initial data satisfies 
$
	(u_0,u_1) \in [H^3(\Omega)\cap H^1_0(\Omega)]
		\times [H^2(\Omega)\cap H^1_0(\Omega)]
$
with the compatibility condition of second order
and
${\rm supp}\,(u_0,u_1)\subset \{ x\in \Omega ; |x| \le R_0 \}$. 
Then for every $t\geq 0$, 
\[
u(x,t)=(e^{tL_*}u_0)(x)-\int_0^te^{(t-s)L_*}\big[a(\cdot)^{-1}u_{tt}(\cdot,s)\big]\,ds,
\]
where
$L_{*}$
is the Friedrichs extension of
$L=a(x)^{-1}\Delta$ in $L^2_{d\mu}$ given in Lemma \ref{L*}.
\end{lemma}
\begin{proof}
In view of \cite[Theorem 4.3.1]{Lunardi}, it suffices to show that 
\begin{align}
\label{sol}
&u\in C^1([0,\infty);L^2_{d\mu})\cap C([0,\infty);D(L_*)), 
\\
\label{inhomo}
&
f=a^{-1}u_{tt}(t)\in C^1([0,\infty);L^2_{d\mu}).
\end{align}
In fact, if the above condition holds, then 
\begin{equation}\label{abstract}
\begin{cases}
w_t+L_*w=-f, \quad t\in [0,\infty), 
\\
w(0)=u_0\in D(L_*)
\end{cases}
\end{equation}
has a unique classical solution in $C^1([0,\infty);L^2_{d\mu})\cap C([0,\infty);D(L_*))$ 
which is represented by 
\[
w(x,t)=(e^{tL_*}u_0)(x)-\int_0^te^{(t-s)L_*}f(\cdot,s)\,ds, \quad t\geq 0. 
\]
Since $u$ also satisfies \eqref{abstract}, it follows from the uniqueness that $u=w$. 

Now we prove \eqref{sol} and \eqref{inhomo}. 
\eqref{sol} is verified by using Lemmas \ref{fp} and \ref{domain.op} 
with 
$u\in C^1([0,\infty);L^2(\Omega))\cap 
C([0,\infty);H^2\cap H_0^1(\Omega))$. 
Finally, we see from 
the regularity $u\in C^3([0,\infty);L^2(\Omega))$ 
and Lemma \ref{fp} that \eqref{inhomo} is satisfied. This completes the proof.
\end{proof}

Now we are in a position to give the proof of Theorem \ref{thm1}. 
\begin{proof}[Proof of Theorem \ref{thm1}]
By the Leibniz rule
\[
	\frac{\partial}{\partial s} \left(  e^{(t-s)L_{\ast}}[ a(\cdot)^{-1}u_{s}(s) ] \right)
	= -L_{\ast}  e^{(t-s)L_{\ast}}[ a(\cdot)^{-1}u_{s}(s) ]
		+ e^{(t-s)L_{\ast}}[ a(\cdot)^{-1}u_{ss}(s) ]
\]
and the fundamental theorem of calculus in Bochner integral,
we have
\begin{align*}
	\int_0^t e^{(t-s)L_{\ast}}[ a(\cdot)^{-1}u_{ss}(\cdot, s) ] ds
	&= \int_{t/2}^t e^{(t-s)L_{\ast}}[ a(\cdot)^{-1}u_{ss}(\cdot, s) ] ds\\
	&\quad +e^{\frac{t}{2}L_{\ast}} [ a(\cdot)^{-1}u_{t}(\cdot, t/2) ]
		- e^{tL_{\ast}}[a(\cdot)^{-1}u_1] \\
	&\quad +\int_0^{t/2} L_{\ast}  e^{(t-s)L_{\ast}}[ a(\cdot)^{-1}u_{s}(\cdot, s) ] ds.
\end{align*}
Combining Lemma \ref{Duhamel} and the above equality yield
\begin{align}
\label{eq_du2}
	u(x,t) - e^{tL_{\ast}}[ u_0 + a(\cdot)^{-1}u_1]
	 &= - \int_{t/2}^t e^{(t-s)L_{\ast}}[ a(\cdot)^{-1}u_{ss}(\cdot, s) ] ds\\
\nonumber
	&\quad -e^{\frac{t}{2}L_{\ast}} [ a(\cdot)^{-1}u_{t}(\cdot, t/2) ] \\
\nonumber
	&\quad -\int_0^{t/2} L_{\ast}  e^{(t-s)L_{\ast}}[ a(\cdot)^{-1}u_{s}(\cdot, s) ] ds.
\end{align}
We estimate the each term of the right-hand side by using
Propositions \ref{embedding2} and \ref{prop_en}.
Hence, we obtain
\begin{align}
\label{eq_du3}
	\| u(\cdot,t) - e^{tL_{\ast}}[ u_0 + a(\cdot)^{-1}u_1] \|_{L^2_{d\mu}}
	\le J_1 + J_2 + J_3,
\end{align}
where
\begin{align*}
	J_1 &= \int^{t}_{t/2}
		\left\| e^{(t-s)L_{\ast}} [ a(\cdot)^{-1}u_{ss}(\cdot, s) ] \right\|_{L^2_{d\mu}} ds,\\
	J_2 &= \left\| e^{\frac{t}{2}L_{\ast}} [ a(\cdot)^{-1}u_{t}(\cdot, t/2) ] \right\|_{L^2_{d\mu}},\\
	J_3 &= \int_0^{t/2}
		\left\| L_{\ast}  e^{(t-s)L_{\ast}}[ a(\cdot)^{-1}u_{s}(\cdot, s) ]\right\|_{L^2_{d\mu}} ds.
\end{align*}
Let us fix
$\ep>0$
so that
$\ep < \frac{2-2\alpha}{2-\alpha}$.
We first estimate $J_1$.
By
$a(x) \ge a_1 (1+|x|)^{-\alpha}$
with some
$a_1>0$
and
\eqref{a1}, we have
\[
	a(x)^{-2} \le C (1+|x|)^{2\alpha}
	\le \left( \frac{ (1+|x|)^{2-\alpha}}{t} \right)^{2\alpha/(2-\alpha)}
		t^{2\alpha/(2-\alpha)}
	\le t^{2\alpha/(2-\alpha)} \Phi_{A,\beta}(x,t).
\]
We use Lemma \ref{L*} and then, Proposition \ref{prop_en} with $k=2$
to obtain
\begin{align*}
	J_1 &\le
		\int^{t}_{t/2}\left\| a(\cdot)^{-1}u_{ss}(\cdot, s) \right\|_{L^2_{d\mu}} ds \\
	&\le \int_{t/2}^t s^{\alpha/(2-\alpha)}
		\left\| \sqrt{\Phi_{A,\beta}}(\cdot, s) u_{ss}(\cdot,s) \right\|_{L^2_{d\mu}} ds \\
	&\le Ct^{\alpha/(2-\alpha)} \int_{t/2}^t
		(t_0+s)^{-\frac{N-\alpha}{2(2-\alpha)} -2+\ep} \| (u_0,u_1) \|_{H^2\times H^1} ds \\
	&\le C (1+t)^{-\frac{N-\alpha}{2(2-\alpha)} - \frac{2-2\alpha}{2-\alpha}+\ep}
		\| (u_0,u_1) \|_{H^2\times H^1}.
\end{align*}
Next, we estimate $J_2$.
In the same way, we have
\begin{align*}
	J_2&\le \left\| a(\cdot)^{-1} u_t(\cdot, t/2) \right\|_{L^2_{d\mu}}\\
	&\le t^{\alpha/(2-\alpha)}
		\left\| \sqrt{\Phi_{A,\beta}}(\cdot, t) u_{t}(\cdot,t/2) \right\|_{L^2_{d\mu}}\\
	&\le (1+t)^{-\frac{N-\alpha}{2(2-\alpha)} - \frac{2-2\alpha}{2-\alpha}+\ep}
		\| (u_0,u_1) \|_{H^2\times H^1}.
\end{align*}
Finally, we estimate
$J_3$.
By Proposition \ref{embedding2}, we have
\begin{align}
\label{j3_es}
	J_3&\le
	\int_0^{t/2} (t-s)^{-\frac{N-\alpha}{2(2-\alpha)}-1}
		\left\| a(\cdot)^{-1} u_s(\cdot,s) \right\|_{L^1_{d\mu}} ds.
\end{align}
The Schwarz inequality and Proposition \ref{en_es} lead to
\begin{align*}
	&\left\| a(\cdot)^{-1} u_s(\cdot,s) \right\|_{L^1_{d\mu}} \\
	&\quad = \int_{\Omega} |u_s(x,s)| dx\\
	&\quad \le \left( \int_{\Omega} \Phi_{A,\beta}(x,s) a(x) |u_s(x,s)|^2 dx \right)^{1/2}
		\left( \int_{\Omega} \Phi_{A,\beta}(x,s)^{-1} a(x)^{-1} dx \right)^{1/2}\\
	&\quad \le C (t_0+s)^{-\frac{N-\alpha}{2(2-\alpha)}-1+\ep}
		s^{\frac{N+\alpha}{2(2-\alpha)}}
		\| (u_0,u_1) \|_{H^2\times H^1}.
\end{align*}
Here we have used
\begin{align*}
	\int_{\Omega} \Phi_{A,\beta}(x,s)^{-1} a(x)^{-1} dx
	&\le C \int_{\Omega} s^{\frac{\alpha}{2-\alpha}}
	\Phi_{A,\beta}(x,s)^{-1} \left( \frac{(1+|x|)^{2-\alpha}}{s} \right)^{\frac{\alpha}{2-\alpha}} dx\\
	&\le C s^{\frac{N+\alpha}{2-\alpha}}.
\end{align*}
Hence, from \eqref{j3_es} we obtain
\begin{align*}
	J_3
	&\le \int_0^{t/2} (t-s)^{-\frac{N-\alpha}{2(2-\alpha)}-1}
		(t_0 + s)^{-1+\frac{\alpha}{2-\alpha}+\ep} \| (u_0,u_1) \|_{H^2\times H^1}\\
	&\le C (1+t)^{-\frac{N-\alpha}{2(2-\alpha)}-\frac{2-2\alpha}{2-\alpha}+\ep}
		\| (u_0,u_1) \|_{H^2\times H^1}.
\end{align*}
These estimates and \eqref{eq_du3} lead the conclusion
\[
	\| u(\cdot,t) - e^{tL_{\ast}}[ u_0 + a(\cdot)^{-1}u_1] \|_{L^2_{d\mu}}
	\le C (1+t)^{-\frac{N-\alpha}{2(2-\alpha)}-\frac{2-2\alpha}{2-\alpha}+\ep}
		\| (u_0,u_1) \|_{H^2\times H^1}.
\]
This completes the proof.
\end{proof}

\section*{Acknowledgement}
The authors are deeply grateful to Professor Mitsuru Sugimoto
for his careful reading of the manuscript and
helpful comments and suggestions.
This work is supported by
Grant-in-Aid for JSPS Fellows 15J01600 of Japan Society for the Promotion of Science.


\end{document}